\documentclass{amsart}

\usepackage{latexsym}
\usepackage{amssymb, latexsym}
\usepackage{amsmath}
\usepackage{mathrsfs}
\usepackage{amsxtra}
\newtheorem{theorem}{Theorem}[section]
\newtheorem{corollary}[theorem]{Corollary}
\newtheorem{sublemma}{Lemma}[theorem]
\newtheorem{lemma}[theorem]{Lemma}

\newtheorem*{main}{Main~Theorem}
\theoremstyle{definition}
\newtheorem{definition}[theorem]{Definition}
\newtheorem{question}[theorem]{Question}
\newtheorem*{scott}{Scott's~Problem}

\newcommand{\x}{\mathfrak{X}}
\newcommand{\y}{\mathfrak{Y}}
\newcommand{\n}{\mathbb {N}}
\newcommand{\p}{\mathbb{P}}

\newcommand{\fin}{\mathrm{Fin}}
\newcommand{\power}{\mathcal{P}}
\newcommand{\pa}{{\rm PA}}
\newcommand{\pfa}{{\rm PFA}}
\newcommand{\ssy}{{\rm SSy}}
\newcommand{\lan}{\mathcal L}

\newcommand{\la}{\langle}
\newcommand{\ra}{\rangle}
\begin{document}
\author{Victoria Gitman}
\today
\address{New York City College of Technology (CUNY),
Mathematics, 300 Jay Street, Brooklyn, NY 11201 USA}
\email{vgitman@nylogic.org}
\title{Scott's problem for Proper Scott sets}
\maketitle
\begin{abstract}
Some 40 years ago, Dana Scott proved that every countable Scott set
is the standard system of a model of \pa.  Two decades later, Knight
and Nadel extended his result to Scott sets of size $\omega_1$.
Here, I show that assuming the Proper Forcing Axiom (\pfa), every
proper Scott set is the standard system of a model of \pa. I define
that a Scott set $\x$ is proper if it is arithmetically closed and
the quotient Boolean algebra $\x/\fin$ is a proper partial order. I
also investigate the question of the existence of proper Scott sets.
\end{abstract}
\section{Introduction}
\long\def\symbolfootnote[#1]#2{\begingroup%
\def\thefootnote{\fnsymbol{footnote}}\footnote[#1]{#2}\endgroup}
\symbolfootnote[0]{I would like to thank my Ph.D. advisor Joel David
Hamkins, Roman Kossak, and Ali Enayat for their many helpful
suggestions.}In this paper, I use the Proper Forcing Axiom (\pfa) to
make partial progress on a half-century-old question in the folklore
of models of Peano Arithmetic about whether every Scott set is the
standard system of a model of \pa. The Proper Forcing Axiom is a
generalization of Martin's Axiom that has found application in many
areas of set theory in recent years. I will show that, assuming
\pfa, every arithmetically closed Scott set whose quotient Boolean
algebra $\x/\fin$ is proper is the standard system of a model of
\pa.

I will begin with some technical and historical details. We can
associate to every model $M$ of {\rm PA} a certain collection of
subsets of the natural numbers called its standard system, in short
$\ssy(M)$.  The natural numbers $\n$ form the initial segment of
every model of ${\rm PA}$. The standard system consists of sets that
arise as intersections of the definable (with parameters) sets of
the model with the \emph{standard} part $\n$. One thinks of the
standard system as the traces left by the definable sets of the
model on the natural numbers.  A different way of characterizing
sets in the standard system uses the notion of coding. Let us say
that a set $A\subseteq \n$ is \emph{coded} in a model $M$ if $M$ has
an element $a$ such that $(a)_n=1$ if and only if $n\in A$. Here,
$(a)_x$ can refer to any of the reader's favorite methods of coding
with elements of a model of \pa, e.g., by defining $(a)_x$ as the
$x^{\text{th}}$ digit in the binary expansion of $a$. It is easy to
see that for \emph{nonstandard} models we can equivalently define
the standard system to be the collection of all subsets of the
natural numbers coded in the model.

What features characterize standard systems?  Without reference to
models of arithmetic, a standard system is just a particular
collection of subsets of the natural numbers.  Can we come up with a
list of elementary (set theoretic, computability theoretic, etc.)
properties that $\x\subseteq \power(\n)$ must satisfy in order to be
the standard system of some model of \pa? The notion of a
\emph{Scott set} encapsulates three key features of standard
systems.

\begin{definition}
$\x\subseteq \power(\n)$ is a \emph{Scott set} if
\begin{itemize}
\item [(1)] $\x$ is a Boolean algebra of sets.
\item [(2)] If $A\in\x$ and $B$ is Turing computable from
$A$, then $B\in\x$. \footnote{Conditions (1) and (2) together imply
that if $A_1,\ldots A_n\in\x$ and $B$ is computable from
$A_1\oplus\cdots \oplus A_n$, then $B\in\x$.}
\item [(3)] If $T$ is an infinite binary tree coded by a set in
$\x$, then $\x$ has a set coding some path through $T$.
\end{itemize}
\end{definition}

It is relatively easy to see that every standard system is a Scott
set (see \cite{kaye:modelsofPA}, p.\ 175). Conversely, Dana Scott
proved in 1962 that every \emph{countable} Scott set is the standard
system of a model of {\rm PA} \cite{scott:ssy}.  The proof relies on
the fact that Scott sets are powerful enough to carry out internally
the Henkin construction used to prove the Completeness Theorem.  A
crucial fact used in the proof is that if a set in the Scott set
codes a consistent theory (again using the reader's favorite
coding), then the Scott set must contain some completion of that
theory as well. This follows easily if one is familiar with the
relationship between building completions of a theory and branches
through binary trees (see \cite{kaye:modelsofPA} p. 177-182 for a
modern version of Scott's proof). So as the first step toward
characterizing standard systems, we have:
\begin{theorem}[Scott, 1962]\label{th:scott}
Every countable Scott set is the standard system of a model of \pa.
\end{theorem}

Thus, countable Scott sets are exactly the countable standard
systems of models of \pa. Scott's theorem leads naturally to the
following folklore question:

\begin{scott}\label{con:Scott}
Is \emph{every} Scott set the standard system of a model of \pa?
\end{scott}

In 1982, Knight and Nadel settled the question for Scott sets of
size $\omega_1$.
\begin{theorem}[Knight and Nadel, 1982]\label{th:scott_omega_1}
Every Scott set of size $\omega_1$ is the standard system of a model
of \pa. \cite{knight:scott}
\end{theorem}
It follows that Scott sets of size $\omega_1$ are exactly the
standard systems of size $\omega_1$ of models of \pa. I will give a
proof of Theorem \ref{th:scott_omega_1} shortly.
\begin{corollary}
If {\rm CH} holds, then Scott's Problem has a positive answer.
\end{corollary}
Very little is known about Scott's Problem if {\rm CH} fails, that
is, for Scott sets of size larger than $\omega_1$. I will use the
techniques of forcing together with forcing axioms to find new
conditions under which given a Scott set we can build a model of
{\rm PA} with that Scott set as the standard system. Given a Scott
set $\x$, we obtain a partial order $\x/\fin$ consisting of the
infinite sets in $\x$ under the ordering of \emph{almost inclusion}.
Let us call a Scott set $\x$ \emph{proper} if it is
\emph{arithmetically closed} and the poset $\x/\fin$ is proper (see
Section \ref{sec:settheory} for the definition of properness and
{\rm PFA} and Section \ref{sec:proof} for the significance of
arithmetic closure). My main theorem is:

\begin{main} Assuming \pfa, every
proper Scott set is the standard system of a model of \pa.
\end{main}
I will prove the main theorem by first generalizing a theorem known
as Ehrenfeucht's Lemma to uncountable models using \pfa.
Ehrenfeucht's Lemma is an unpublished result of Ehrenfeucht from
around the 1970's \footnote{Roman Kossak, personal communication.}.
One can use Ehrenfeucht's Lemma to give a proof of Theorem
\ref{th:scott_omega_1}. Similarly, my generalized version is used to
prove the main theorem. It is instructive to see a proof of
Ehrenfeucht's Lemma to understand the difficulties involved in
extending it to uncountable models (the proof below follows
\cite{smorynski:ssy}).
\begin{theorem}[Ehrenfeucht's Lemma]\label{th:ehrenlemma}
If $M$ is a countable model of {\rm PA} whose standard system is
contained in a Scott set $\x$, then for any $A\in\x$ there is an
elementary extension $M\prec N$ such that $A\in {\rm
SSy}(N)\subseteq \x$.
\end{theorem}
\begin{proof} First, we consider nonstandard $M$. Let $\x$ be a Scott
set such that $\ssy(M)\subseteq\x$ and let $A\in\x$. Choose a
countable Scott set $\y\subseteq \x$ containing $\ssy(M)$ and $A$.
Using the truth predicate for $\Sigma_n$-formulas, we can prove that
the $\Sigma_n$-theory of $M$ is coded by a set in $\ssy(M)$ for
every $n$. Moreover, all computable theories are in $\ssy(M)$ since
$\ssy(M)$ is a Scott set. Therefore, the theory
$T$:=``$\pa+\Sigma_1$-theory of $M$" is in $\ssy(M)$. The idea that
the Henkin construction can be carried out inside a Scott gives more
than just Theorem \ref{th:scott}. From this it follows that for any
theory $T\supseteq \pa$ such that $T\cap \Sigma_n\in \x$ for every
$n\in\n$, there exists a model of $T$ with that Scott set as the
standard system. In particular, we get a model $M^*$ of $T$ with
$\ssy(M^*)=\y$.  By Friedman's Embedding Theorem (see
\cite{kaye:modelsofPA}, p.\ 161), since
$M^*\models``\Sigma_1$-theory of $M$" and $\ssy(M)\subseteq \y$, we
have $M\prec_{\Delta_0} M^*$. Close $M$ under initial segment and
call the resulting submodel $N$. Then $M\prec N$ since it is cofinal
and $\Delta_0$-elementary (by Gaifman's Embedding Theorem, see
\cite{kaye:modelsofPA}, p.\ 87). But also $\ssy(N)=\ssy(M^*)=\y$ as
required since $N$ is an initial segment of $M^*$. This completes
the proof for nonstandard models. Let ${\rm
TA}:=\{\varphi\mid\n\models\varphi\}$ denote \emph{True Arithmetic}.
It is clear that $\n\prec N$ if and only if $N\models{\rm TA}$. The
standard system of $\n$ is the collection of all arithmetic sets. So
suppose that $\x$ is a Scott set containing all arithmetic sets and
fix $A\in\x$. It follows that ${\rm TA}\cap \Sigma_n$ is in $\x$ for
every $n\in\n$. Let $\y\subseteq\x$ be a countable Scott set
containing $A$ and ${\rm TA}\cap \Sigma_n$ for every $n\in\n$. By
the remark above, there exists a model $N\models{\rm TA}$ whose
standard system is exactly $\y$. Thus, $\n\prec N$ and $A\in$ ${\rm
SSy}(N)\subseteq\x$.
\end{proof}
We are now ready to prove Knight and Nadel's result. \footnote{This
is not Knight and Nadel's original proof.}
\begin{proof}
[Proof of Theorem \ref{th:scott_omega_1}] Let $\x$ be a Scott set of
size $\omega_1$ and enumerate $\x=\{A_\xi\mid\xi<\omega_1\}$. The
idea is to build up a model with the Scott set $\x$ as the standard
system in $\omega_1$ steps by successively throwing in one more set
at each step and using Ehrenfeucht's Lemma to stay within $\x$. More
precisely, we will define an elementary chain $M_0\prec
M_1\prec\cdots \prec M_\xi\prec\cdots$ of length $\omega_1$ of
countable models of {\rm PA} such that ${\rm SSy}(M_\xi)\subseteq\x$
and $A_\xi\in$ ${\rm SSy}(M_{\xi+1})$. Then clearly
$M=\cup_{\xi<\omega_1}M_\xi$ will work. Let $M_0$ be any countable
model of {\rm PA} with ${\rm SSy}(M_0)\subseteq \x$. Such $M_0$
exists by Scott's theorem (\ref{th:scott}). Given $M_\xi$, by
Ehrenfeucht's Lemma, there exists $M_{\xi+1}$ such that $M_\xi\prec
M_{\xi+1}$, the set \hbox{$A_\xi\in{\rm SSy}(M_{\xi+1})$}, and ${\rm
SSy}(M_{\xi+1})\subseteq \x$. At limit stages take unions.
\end{proof}

The key ideas in the proof of Theorem \ref{th:scott_omega_1} can be
summarized in the following definition and theorem:
\begin{definition}[The $\kappa$-Ehrenfeucht Principle for $\Gamma$]\label{def:ehren}
Let $\kappa$ be a cardinal and $\Gamma$ some collection of Scott
sets.  The \emph{$\kappa$-Ehrenfeucht Principle for $\Gamma$} states
that if $M$ is a model of {\rm PA} of size less than $\kappa$ and
$\x$ is a Scott set in $\Gamma$ such that $\ssy(M)\subseteq \x$,
then for any $A\in\x$ there is an elementary extension $M\prec N$
such that $A\in \ssy(N)\subseteq \x$.  If $\Gamma$ is the collection
of all Scott sets, we will say simply that the $\kappa$-Ehrenfeucht
Principle holds.
\end{definition}
In view of Definition \ref{def:ehren}, Ehrenfeucht's Lemma (Theorem
\ref{th:ehrenlemma}) is the\break $\omega_1$-Ehrenfeucht Principle.
We can freely assume that the elementary extension $N$ given by the
$\kappa$-Ehrenfeucht Principle has size less than $\kappa$ since if
this is not the case, we can always take an elementary submodel $N'$
of $N$ such that $M\prec N'$ and $A\in \ssy(N')$. A completely
straightforward generalization of the proof of Theorem
\ref{th:scott_omega_1} gives:
\begin{theorem}\label{th:ehrenprin}
If the $\kappa$-Ehrenfeucht Principle for $\Gamma$ holds, then every
Scott set in $\Gamma$ of size $\kappa$ is the standard system of a
model of \pa.
\end{theorem}

Thus, one approach to solving Scott's Problem would be to try to
prove the $\omega_2$-Ehrenfeucht Principle for some collection of
Scott sets. However, proofs of Ehrenfeucht's Lemma hinge precisely
on those techniques in the field of models of {\rm PA} that appear
to work only with countable models. As an example, Friedman's famous
Embedding Theorem does not generalize to uncountable
models.\footnote{Here $\omega_1$-like models are an obvious
counterexample.} In what follows, I will mainly investigate the
Ehrenfeucht principles. The results on Scott's Problem will follow
as a corollary. Under {\rm PFA}, I will show that the
\emph{$\omega_2$-Ehrenfeucht Principle for proper Scott sets} holds
(Theorem \ref{th:ehren}).

\section {Set Theory and Scott's
Problem}\label{sec:settheory} Since the result of Knight and Nadel
(Theorem \ref{th:scott_omega_1}), very little progress had been made
on Scott's Problem until some recent work of Fredrik Engstr\"om
\cite{engstrom:thesis}. It is not difficult to believe that Scott's
Problem past $\omega_1$ might have a set theoretic resolution.
Engstr\"om followed a strategy, suggested more than a decade earlier
by Joel Hamkins and others, to use forcing axioms to gain new
insight into the problem. We saw that a positive answer to Scott's
Problem follows from {\rm CH}. It is a standard practice in set
theoretic proofs that if a statement follows from {\rm CH}, we try
to prove it or its negation from $\neg{\rm CH}$ + Martin's Axiom.
Martin's Axiom ({\rm MA}) is a forcing axiom which asserts that for
every c.c.c.\ poset $\p$ and every collection $\mathcal{D}$ of less
than the continuum many dense subsets of $\p$, there is a filter on
$\p$ that meets all of them. Such filters are often called
\emph{partially generic filters}. Engstr\"om tried to use Martin's
Axiom to find new techniques for building models of {\rm PA} whose
standard system is a given Scott set.

Given a Scott set $\x$, Engstr\"om chose the poset $\x/\fin$, whose
elements are infinite sets in $\x$ ordered by \emph{almost
inclusion}. That is, for infinite $A$ and $B$ in $\x$, we say that
$A\leq B$ if and only if $A\subseteq_\fin B$. Observe that $\x/\fin$
is forcing equivalent to forcing with the Boolean algebra $\x$
modulo the ideal of finite sets. A familiar and thoroughly studied
instance of this poset is $\power(\n)/\fin$. A Scott set is
\emph{arithmetically closed} if whenever $A$ is in it and $B$ is
arithmetically definable from $A$, then $B$ is also in it (for a
more extensive discussion, see Section \ref{sec:proof}). For a
property of posets $\mathscr P$, if $\x$ is an arithmetically closed
Scott set and $\x/\fin$ has $\mathscr P$, I will simply say that
\emph{$\x$ has property $\mathscr P$}. An important point to be
noted here is that whenever a Scott set $\x$ is discussed as a
poset, I will always be assuming that \emph{it is arithmetically
closed}. The significance of arithmetic closure will become apparent
in Section \ref{sec:proof}.
\begin{theorem} [Engstr\"om, 2004]\label{th:engs}
Assuming Martin's Axiom, every c.c.c.\ Scott set of size less than
the continuum is the standard system of a model of {\rm PA}.
\cite{engstrom:thesis}
\end{theorem}

To obtain models for Scott sets for which we could not do so before,
Engstr\"om needed that there are uncountable Scott sets that are
c.c.c.. Unfortunately:
\begin{theorem}\label{th:wrong} Every c.c.c.\ Scott set is countable.
\end{theorem}
\begin{proof}
Let $\x$ be a Scott set.  If $x$ is a finite subset of $\n$, let
$\ulcorner x \urcorner$ denote the code of $x$ using G\"{o}del's
coding. For every $A\in\x$, define an associated $A'=\{\ulcorner
A\cap n\urcorner\mid n\in\n\}$. Clearly $A'$ is computable from $A$,
and hence in $\x$. Observe that if $A\neq B$, then $|A'\cap
B'|<\omega$. Hence if $A\neq B$, we get that $A'$ and $B'$ are
incompatible in $\x/\fin$. It follows that $\mathscr A=\{A'\mid
A\in\x\}$ is an antichain of $\x/\fin$ of size $|\x|$.  This shows
that $\x/\fin$ always has antichains as large as the whole poset.
\end{proof}

Thus, the poset $\x/\fin$ has the worst possible chain condition,
namely $|\x|^+$-c.c.. Theorem \ref{th:wrong} implies that no new
instances of Scott's Problem can be obtained from Theorem
\ref{th:engs}.

I will borrow from Engstr\"om's work the poset $\x/\fin$. But my
strategy will be different in two respects. First, instead of {\rm
MA}, I will use the poset together with the forcing axiom {\rm PFA},
allowing me to get around the obstacle of Theorem \ref{th:wrong}. In
Section \ref{sec:proper}, I will argue that, unlike the case with
c.c.c.\ Scott sets, uncountable proper Scott sets do exist. However,
I will not be able to explicitly obtain any new instances of Scott's
Problem. Second, my main aim will be to obtain an extension of
Ehrenfeucht's Lemma to uncountable models, while Engstr\"om's was to
directly get a model whose standard system is a given Scott set.
This approach will allow me to handle Scott sets of size continuum,
which had not been possible with Engstr\"om's techniques.

Recall that for a cardinal $\lambda$, the set $H_\lambda$ is the
collection of all sets whose transitive closure has size less than
$\lambda$. Let $\p$ be a poset and $\lambda$ be a cardinal greater
than $2^{|\p|}$. Since we can always take an isomorphic copy of $\p$
on the cardinal $|\p|$, we can assume without loss of generality
that $\p$ and $\power(\p)$ are elements of $H_\lambda$. In
particular, we want to ensure that all dense subsets of $\p$ are in
$H_\lambda$. Let $M$ be a countable elementary submodel of
$H_\lambda$ containing $\p$ as an element. If $G$ is a filter on
$\p$, we say that $G$ is $M$-\emph{generic} if for every maximal
antichain $A\in M$ of $\p$, the intersection $G\cap A\cap
M\neq\emptyset$. It must be explicitly specified what $M$-generic
means in this context since the usual notion of generic filters
makes sense only for transitive structures and $M$ is not
necessarily transitive. This definition of $M$-generic is closely
related to the definition for transitive structures. To see this,
let $M^*$ be the Mostowski collapse of $M$ and $\p^*$ be the image
of $\p$ under the collapse. Let $G^*\subseteq \p^*$ be the pointwise
image of $G\cap M$ under the collapse. Then $G$ is $M$-generic if
and only if $G^*$ is $M^*$-generic for $\p^*$ in the usual sense.

\begin{definition} Let $\p\in H_\lambda$ be a poset and $M$ be an
elementary submodel of $H_\lambda$ containing $\p$. Then a condition
$q\in\p$ is $M$-\emph{generic} if and only if every $V$-generic
filter $G\subseteq\p$ containing $q$ is $M$-generic.
\end{definition}
\begin{definition} A poset $\p$ is \emph{proper} if for every
$\lambda> 2^{|\p|}$ and every countable $M\prec H_\lambda$
containing $\p$, for every $p\in\p\cap M$, there is an $M$-generic
condition below $p$.
\end{definition}
It can be shown that it is actually equivalent to consider only some
fixed $\lambda>2^{|\p|}$ and to show that generic conditions exist
only for a club of countable $M\prec H_\lambda$ \cite{shelah:proper}
(p.\ 102).
\begin{definition}
\emph{The Proper Forcing Axiom} ({\rm PFA}) is the assertion that
for every proper poset $\p$ and every collection $\mathcal D$ of at
most $\omega_1$ many dense subsets of $\p$, there is a filter on
$\p$ that meets all of them.
\end{definition}
Proper forcing was invented by Shelah, who sought a class of
$\omega_1$-preserving forcing notions that would be preserved under
countable support iterations (for an introduction to proper forcing
see \cite{jech:settheory} (p. 601) or \cite{shelah:proper}).  The
two familiar classes of $\omega_1$-preserving forcing notions,
namely the c.c.c.\ and countably closed forcing notions, turn out to
be proper as well. The Proper Forcing Axiom, introduced by
Baumgartner \cite{baumgartner:pfa}, is easily seen to be a
generalization of Martin's Axiom since c.c.c.\ posets are proper and
{\rm PFA} decides the size of the continuum is $\omega_2$. The later
fact is a highly nontrivial result in \cite{veli:pfa}. In many
respects, however, {\rm PFA} is very much unlike {\rm MA}. Not only
does it decide the size of the continuum, the axiom also has large
cardinal strength. The best known large cardinal upper bound on the
consistency of {\rm PFA} is a supercompact cardinal
\cite{baumgartner:pfa}. Much fruitful set theoretical work in recent
years has involved {\rm PFA} and its consequences.

\section {Proof of the Main Theorem}\label{sec:proof}
I will use {\rm PFA} to prove the \emph{$\omega_2$-Ehrenfeucht
Principle for proper Scott sets}. The main theorem will follow as a
corollary.

A filter $G$ on the poset $\x/\fin$ is easily seen to be a filter on
the Boolean algebra $\x$. By extending $G$ to a larger filter if
necessary, we can assume without loss of generality that $G$ is an
ultrafilter. Recall that to prove the $\omega_2$-Ehrenfeucht
Principle, given a model $M$ of size $\omega_1$ and a Scott set $\x$
such ${\rm SSy(M)}\subseteq \x$, we need to find for every $A\in\x$,
an elementary extension $N$ such that $A\in {\rm SSy}(N)\subseteq
\x$. The strategy will be to find $\omega_1$ many dense subsets of
$\x/\fin$ such that if $G$ is a partially generic ultrafilter
meeting all of them, then the standard system of the ultrapower of
$M$ by $G$ will stay within $\x$.  Thus, if $\x$ is proper, we will
be able to use {\rm PFA} to obtain such an ultrafilter. I will also
show that to every $A\in\x$, there corresponds a set $B\in\x/\fin$
such that whenever $B$ is in an ultrafilter $G$, the set $A$ will
end up in the ultrapower of $M$ by $G$.

Let $\mathcal S\subseteq \power(\n)$ and expand the language of
arithmetic $\lan_A$ to include unary predicates for all $A\in
\mathcal S$.  Then the structure $\n_\mathcal S=\langle \n,
A\rangle_{A\in\mathcal S}$ is a structure of this expanded language
with the natural interpretation. Since Scott sets are closed under
relative computability, basic computability theory arguments show
that if $\x$ is a Scott set, the structure $\n_{\x}=\langle \n,
A\rangle_{A\in \x}$ is closed under $\Delta_1$-definability. That
is, if $B$ is $\Delta_1$-definable in $\n_\x$, then $B\in \x$.

\begin{definition}\label{def:closed}
A collection $\mathcal S\subseteq\power(\n)$ is \emph{arithmetically
closed} if the structure $\n_\mathcal S=\langle \n,
A\rangle_{A\in\mathcal S}$ is closed under definability. That is, if
$B$ is definable in $\n_\mathcal S$, then $B\in\mathcal S$.
\end{definition}
A Scott set $\x$ is \emph{arithmetically closed} simply when it
satisfies Definition \ref{def:closed}. Observe actually that if
$\mathcal S$ is arithmetically closed, then it is a Scott set. Thus,
arithmetic closure subsumes the definition of a Scott set. An easy
induction on the complexity of formulas establishes that if $\x$ is
a Boolean algebra of sets and $\n_\x=\langle \n, A\rangle_{A\in \x}$
is closed under $\Sigma_1$-definability, then $\x$ is arithmetically
closed. Hence a Scott set is arithmetically closed if and only if it
is closed under the Turing jump operation.

\begin{definition}
Say that $\langle B_n\mid n\in\n\rangle$ is \emph{coded} in $\x$ if
there is $B\in \x$ such that $B_n=\{m\in \n\mid \langle n,m\rangle
\in B\}$. Given $\langle B_n\mid n\in\n\rangle$ coded in $\x$ and
$C\in\x/\fin$, say that $C$ \emph{decides} $\langle B_n\mid
n\in\n\rangle$ if whenever $U$ is an ultrafilter on $\x$ and $C\in
U$, then $\{n\in\n\mid B_n\in U\}\in \x$. Call a Scott set $\x$
\emph{decisive} if for every $\langle B_n\mid n\in \n\rangle$ coded
in $\x$, the set $\mathscr D=\{C\in\x/\fin\mid C\text{
decides}\break\langle B_n\mid n\in\n\rangle\}$ is dense in
$\x/\fin$.
\end{definition}

Decisiveness is precisely the property of a Scott set which is
required for our proof of the main theorem. I will show below that
decisiveness is equivalent to arithmetic closure.
\begin{lemma}\label{th:decisive}
The following are equivalent for a Scott set $\x$:
\begin{itemize}
\item[(1)] $\x$ is arithmetically closed.
\item[(2)] $\x$ is decisive.
\item[(3)] For every sequence $\la B_n\mid
n\in\n\ra$ coded in $\x$, there is $C\in\x/\fin$ deciding $\la
B_n\mid n\in\n\ra$.
\end{itemize}
\end{lemma}
\begin{proof}
$\,$\\
(1)$\Longrightarrow$(2):\footnote{Similar arguments have appeared in
\cite{engstrom:thesis} and other places.} Assume that $\x$ is
arithmetically closed. Fix $A\in\x/\fin$ and a sequence $\langle
B_n\mid n\in\n\rangle$ coded in $\x$. We need to show that there is
an element in $\x/\fin$ below $A$ deciding $\langle B_n\mid
n\in\n\rangle$. For every finite binary sequence $s$, we will define
$B_s$ by induction on the length of $s$. Let $B_\emptyset=A$. Given
$B_s$, where $s$ has length $n$, define $B_{s1}=B_s\cap B_n$ and
$B_{s0}=B_s\cap (\n-B_n)$. Define the binary tree $T=\{s\in
2^{<\omega}\mid B_s\text{ is infinite}\}$. Clearly $T$ is infinite
since if we split an infinite set into two pieces one of them must
still be infinite. Since $\x$ is arithmetically closed and $T$ is
arithmetic in $A$ and $\langle B_n\mid n\in \n\rangle$, it follows
that $T\in\x$. Thus, $\x$ contains a cofinal branch $P$ through $T$.
Define $C=\{b_n\mid n\in\n\}$ such that $b_0$ is least element of
$B_\emptyset$ and $b_{n+1}$ is least element of $B_{P\upharpoonright
n}$ that is greater than $b_n$. Clearly $C$ is infinite and
$C\subseteq A$. Now suppose $U$ is an ultrafilter on $\x$ and $C\in
U$, then $B_n\in U$ if and only if $C\subseteq_\fin B_n$. Thus,
$\{n\in\n\mid B_n\in U\}=\{n\in \n\mid
C\subseteq_\fin B_n \}\in \x$ since $\x$ is arithmetically closed.\\
(2)$\Longrightarrow$(3):
Clear.\\
(3)$\Longrightarrow$(1)\footnote{I am grateful to Joel Hamkins for
pointing out this argument.}: It suffices to show that $\x$ is
closed under the Turing jump operation. Fix $A\in\x$ and define the
sequence $\la B_n\mid n\in\n\ra$ by $k\in B_n$ if and only if the
Turing program coded by $n$ with oracle $A$ halts on input $n$ in
less than $k$ many steps. Clearly the sequence is computable from
$A$, and hence coded in $\x$. Let $H=\{n\in\n\mid$ the program coded
by $n$ with oracle $A$ halts on input $n\}$ be the halting problem
for $A$. It should be clear that $n\in H$ implies that $B_n$ is
cofinite and $n\notin H$ implies that $B_n=\emptyset$. Let $C\in
\x/\fin$ deciding $\la B_n\mid n\in\n\ra$ and $U$ be any ultrafilter
containing $C$, then $\{n\in\n\mid B_n\in U\}\in \x$. But this set
is exactly $H$. This shows that $H\in \x$, and hence $\x$ is closed
under the Turing jump operation.
\end{proof}
\begin{theorem}\label{th:ehren}
Assuming {\rm PFA}, the $\omega_2$-Ehrenfeucht Principle for proper
Scott sets holds. That is, if $\x$ is a proper Scott set and $M$ is
a model of\/ {\rm PA} of size $\omega_1$ whose standard system is
contained in $\x$, then for any $A\in\x$, there is an elementary
extension $M\prec N$ such that $A\in {\rm SSy}(N)\subseteq \x$.
\end{theorem}
\begin{proof}
I will build $N$ using a variation on the ultrapower construction
introduced by Kirby and Paris \cite{kirby:ults}. Fix a model $M$ of
{\rm PA} and a Scott set $\x$ such that {\rm SSy}$(M)\subseteq\x$.
Let $G$ be some ultrafilter on $\x$. If $f:\n\to M$, we say that $f$
is \emph{coded} in $M$ when there is $a\in M$ such that $(a)_n=f(n)$
for all $n\in\n$. Given $f$ and $g$ coded in $M$, define $f\sim_G g$
if $\{n\in\n\mid f(n)=g(n)\}\in G$. The definition makes sense since
clearly $\{n\in\n\mid f(n)=g(n)\}\in$ {\rm SSy}$(M)\subseteq\x$. The
classical ultrapower construction uses an ultrafilter on
$\power(\n)$ and all functions from $\n$ to $M$. This construction
uses only functions coded in $M$, and therefore needs only an
ultrafilter on {\rm SSy}$(M)\subseteq\x$.  As in the classical
construction, we get an equivalence relation and a well-defined
$\lan_A$ structure on the equivalence classes. The proof relies on
the fact that $\x$ is a Boolean algebra. Call $\Pi_\x M/G$ the
collection of equivalence classes $[f]_G$ where $f$ is coded in $M$.
Also, as usual, we get:
\begin{sublemma}
\L o\'{s} Lemma holds.  That is, $\Pi_\x M/G\models \varphi([f]_G)$
if and only if $\{n\in\n\mid M\models \varphi(f(n))\}\in G$.
\end{sublemma}
\begin{proof}
Similar to the classical proof of the \L o\'{s} Lemma.
\end{proof}

\begin{sublemma}\label{l:add}
For every $A\in\x$, there is $B\in \x/\fin$ such that if $G$ is any
ultrafilter on $\x$ containing $B$, then $A\in {\rm SSy}(\Pi_\x
M/G)$.
\end{sublemma}
\begin{proof}
Let $\chi_A$ be the characteristic function of $A$.  For every
$n\in\n$, define $B_n=\{m\in\n\mid (m)_n=\chi_A(n)\}$. Then clearly
each $B_n\in \x$ and $\langle B_n\mid n\in\n\rangle$ is coded in
$\x$ since the sequence is arithmetic in $A$.  Observe that the
intersection of any finite number of $B_n$ is infinite. Let
$B=\{b_n\mid n\in\n\}$ where $b_0$ is least element of $B_0$ and
$b_{n+1}$ is least element of $\cap_{m\leq n+1}B_m$ that is greater
than $b_n$. Then $B\subseteq_\fin B_n$ for all $n\in\n$ and $B\in
\x$ since it is arithmetic in $\la B_n:n\in\n\ra$. It follows that
if $G$ is any ultrafilter containing $B$, then $G$ must contain all
the $B_n$ as well. Let $G$ be an ultrafilter containing $B$. Let
$id:\n\to \n$ be the identity function. I claim that
$([id]_G)_n=\chi_A(n)$. It will follow that $A\in {\rm SSy}(\Pi_\x
M/G)$. But this is true since $([id]_G)_n=\chi_A(n)$ if and only if
$\{m\in\n\mid (m)_n=\chi_A(n)\}=B_n\in G$.
\end{proof}

Lemma \ref{l:add} tells us that if we want to add some set $A$ to
the standard system of the ultrapower that we are building, we just
have to make sure that a correct set gets put into the ultrafilter.
It follows that that we can build ultrapowers of $M$ having any
given element of $\x$ in the standard system.

The crucial step of the construction is to find a family of size
$\omega_1$ of dense subsets of $\x/\fin$ such that if the
ultrafilter meets all members of the family, the standard system of
the ultrapower stays within $\x$. It is in this step that we need
the \emph{decisiveness} of $\x$.

Recall that a set $E$ is in the standard system of a nonstandard
model if and only if there is an element $e$ such that
$E=\{n\in\n\mid (e)_n=1$\}, meaning $E$ is \emph{coded} in the
model. Thus, we have to show that the sets coded by elements of
$\Pi_\x M/G$ are in $\x$.
\begin{sublemma}
For every function $f:\n\to M$ coded in $M$, there is a dense subset
$\mathscr D_f$ of\/ $\x/\fin$ such that if $G$ meets $\mathscr D_f$,
then $[f]_G\in \Pi_\x M/G$ codes a set in $\x$.
\end{sublemma}
\begin{proof}
Fix a function $f:\n\to M$ coded in $M$ and let $E_f=\{n\in\n\mid
\Pi_\x M/G\models ([f]_G)_n=1\}$. By \L o\'{s} Lemma, $\Pi_\x
M/G\models ([f]_G)_n=1$ if and only if\break $\{m\in \n\mid
(f(m))_n=1\}\in G$. Define $B_{n,f}=\{m\in\n\mid (f(m))_n=1\}$ and
note that $\langle B_{n,f}\mid n\in \n\rangle$  is coded in ${\rm
SSy}(M)$. Observe that $n\in E_f$ if and only if $B_{n,f}\in G$.
Thus, we have to make sure that $\{n\in \n\mid B_{n,f}\in G\}\in
\x$. Let $\mathscr D_f=\{C\in \x/\fin\mid C \text{ decides }\langle
B_{n,f}\mid n\in\n\rangle\}$. Since $\x$ is decisive, $\mathscr D_f$
is dense. Clearly if
 $G$ meets $\mathscr
D_f$, the set coded by $[f]_G$ will be in $\x$.
\end{proof}

Now we can finish the proof of Theorem \ref{th:ehren}.   Let
$\mathcal{D}=\{\mathscr D_f\mid f:\n\to M$ is coded in $M\}$. Since
$M$ has size $\omega_1$, the collection $\mathcal{D}$ has size
$\omega_1$ also. Assuming {\rm PFA} guarantees that we can find an
ultrafilter $G$ meeting every $\mathscr D_f\in \mathcal{D}$. But
this is precisely what forces the standard system of $\Pi_\x M/G$ to
stay inside $\x$.
\end{proof}
The main theorem now follows directly from Theorem \ref{th:ehren}.
\begin{proof} [Proof of Main Theorem] Since {\rm PFA}
implies $2^\omega=\omega_2$ and Scott sets of size $\omega_1$ are
already handled by Knight and Nadel's result, we only need to
consider Scott sets of size $\omega_2$. But the result for these
follows from Theorem \ref{th:ehrenprin} and the
$\omega_2$-Ehrenfeucht Principle established by Theorem
\ref{th:ehren}.
\end{proof}

\section{Extensions of Ehrenfeucht's Lemma} Below, I will go through some results related to the
question of extending Ehrenfeucht's Lemma to models of size
$\omega_1$ ($\omega_2$-Ehrenfeucht Principle).

Theorem \ref{th:ehren} shows that in a universe satisfying {\rm
PFA}, the \emph{$\omega_2$-Ehrenfeucht Principle for proper Scott
sets} holds. Next, I will use the same techniques to show that the
\emph{$\kappa$-Ehrenfeucht Principle for arithmetically closed Scott
sets} holds for all $\kappa$ if we only consider models with
\emph{countable} standard systems. For this argument, we do not need
to use {\rm PFA} or properness.
\begin{theorem} \label{th:countehren}
If $M$ is a model of {\rm PA} whose standard system is countable and
contained in an arithmetically closed Scott set $\x$, then for any
$A\in\x$, there is an elementary extension $M\prec N$ such that
$A\in {\rm SSy}(N)\subseteq \x$.
\end{theorem}
\begin{proof}
Fix an arithmetically closed Scott set $\x$ and a model $M$ of {\rm
PA} such that ${\rm SSy}(M)$ is countable and contained in $\x$. To
mimic the proof of Theorem \ref{th:ehren}, we need to find an
ultrafilter $G$ on $\x$ which meets the dense sets $\mathscr
D_f=\{C\in \x/\fin\mid C \text{ decides }\langle B_{n,f}\mid
n\in\n\rangle\}$. I claim that there are only countably many
$\mathscr D_f$. If this is the case, then such an ultrafilter exists
without any forcing axiom assumption. Given $f:\n\to M$, let $B_f$
code $\langle B_{n,f}\mid n\in\n\rangle$. There are possibly as many
$f$ as elements of $M$, but there can be only countably many $B_f$
since each $B_f\in {\rm SSy}(M)$. It remains only to observe that
$\mathscr D_f$ is determined by $B_f$. So there are as many
$\mathscr D_f$ as there are different $B_f$. Thus, there are only
countably many $\mathscr D_f$ in spite of the fact that $M$ can be
arbitrarily large.
\end{proof}
The same idea can be used to extend Theorem \ref{th:ehren} to show
that the \emph{$\kappa$-Ehrenfeucht Principle for proper Scott sets}
holds for all $\kappa$ if we consider only models whose standard
system has size $\omega_1$.
\begin{theorem}
Assuming {\rm PFA}, if\/ $\x$ is a proper Scott set and $M$ is a
model of {\rm PA} whose standard system has size $\omega_1$ and is
contained in $\x$, then for any $A\in \x$, there is an elementary
extension $M\prec N$ such that $A\in {\rm SSy}(N)\subseteq \x$.
\end{theorem}

It is also an easy consequence of an amalgamation result for models
of {\rm PA} that the $\kappa$-Ehrenfeucht Principle holds for all
$\kappa$ for models with a \emph{countable nonstandard elementary
initial segment}. Neither {\rm PFA} nor arithmetic closure is
required for this result.
\begin{theorem}
Suppose $M_0$, $M_1$, and $M_2$ are models of {\rm PA} such that
$M_0\prec_{cof} M_1$  and $M_0\prec_{end}M_2$.  Then there is an
amalgamation $M_3$ of $M_1$ and $M_2$ over $M_0$ such that
$M_1\prec_{end}M_3$ and $M_2\prec_{cof}M_3$. \emph{(See
\cite{kossak:book}, p. 40)}
\end{theorem}
\begin{theorem}\label{th:amal}
Suppose $M$ is a model of {\rm PA} with a countable nonstandard
elementary initial segment and $\x$ is a Scott set such that ${\rm
SSy}(M)\subseteq \x$. Then for any $A\in\x$, there is an elementary
extension $M\prec N$ such that $A\in {\rm SSy}(N)\subseteq \x$.
\end{theorem}
\begin{proof}
Fix a set $A\in \x$. Let $K$ be a countable nonstandard elementary
initial segment of $M$, then ${\rm SSy}(K) = {\rm SSy}(M)$. By
Ehrenfeucht's Lemma (Theorem \ref{th:ehrenlemma}), there is an
extension $K\prec_{cof} K'$ such that $A\in {\rm SSy}(K')\subseteq
\x$.  By Theorem \ref{th:amal}, there is a model $N$, an
amalgamation of $K'$ and $M$ over $K$, such that $K'\prec_{end} N$
and $M\prec_{cof}N$. It follows that ${\rm SSy}(K') = {\rm SSy}(N)$.
Thus, $A\in {\rm SSy}(N)\subseteq \x$.
\end{proof}
\begin{corollary}
The $\kappa$-Ehrenfeucht Principle holds for $\omega_1$-like models
for all cardinals $\kappa$.
\end{corollary}
These observations suggest that if the $\omega_2$-Ehrenfeucht
Principle fails to hold, one should look to models with an
uncountable standard system for such a counterexample.

\section{Other Applications of $\x/\fin$}
It appears that $\x/\fin$ is a natural poset to use in several
unresolved questions in the field of models of {\rm PA}.  In the
previous sections, I used it to find new conditions for extending
Ehrenfeucht's Lemma and Scott's Problem. Here, I will mention some
other instances in which the poset naturally arises.

\begin{definition}
Let $\lan$ be some language extending the language of arithmetic
$\lan_A$. We say that a model $M$ of $\lan$ satisfies {\rm PA}$^*$
if $M$ satisfies induction axioms in the expanded language. If
$M\models {\rm PA}^*$, then $M\subseteq N$ is a \emph{conservative
extension} if it is a proper extension and every parametrically
definable subset of $N$ when restricted to $M$ is also definable in
$M$.
\end{definition}
Gaifman  showed in \cite{gaifman:types} that for any countable
language $\lan$, every $M\models{\rm PA}^*$ in $\lan$ has a
conservative elementary extension.  A result of George Mills shows
that the statement fails for uncountable languages. Mills proved
that every countable nonstandard model $M\models {\rm PA}^*$ in a
countable language has an expansion to an uncountable language such
that $M\models{\rm PA}^*$ in the expanded language, but has no
conservative elementary extension (see \cite{kossak:book}, p. 168).
His techniques failed for the standard model, leaving open the
question whether there is an expansion of the standard model $\n$ to
some uncountable language that does not have a conservative
elementary extension. This question has recently been answered by
Ali Enayat, who demonstrated that there is always an uncountable
arithmetically closed Scott set $\x$ such that $\langle \n,
A\rangle_{A\in\x}$ has no conservative elementary extension
\cite{enayat:endextensions}. This raises the question of whether we
can say something general about Scott sets $\x$ for which $\langle
\n, A\rangle_{A\in\x}$ has a conservative elementary extension.
\begin{theorem}\label{f:conservative}
Assuming {\rm PFA}, if\/ $\x$ is a proper Scott set of size
$\omega_1$, then $\langle \n, A\rangle_{A\in\x}$ has a conservative
elementary extension.
\end{theorem}
\begin{proof}
Let $\lan_\x$ be the language of arithmetic $\lan_A$ together with
unary predicates for sets in $\x$. Let $G$ be an ultrafilter on
$\x$. We define $\Pi_\x\n/G$, the ultrapower of $\n$ by $G$, to
consist of equivalence classes of functions coded in $\x$. We have
to make this modification to the construction of the proof of
Theorem \ref{th:ehren} since the idea of functions coded in the
model clearly does not make sense for $\n$. The usual arguments show
that we can impose an $\lan_\x$ structure on $\Pi_\x\n/G$ and \L
o\'{s} Lemma holds. I will show, by choosing $G$ carefully, that
$\la \Pi_\x\n/G, A'\ra_{A\in \x}$ is a conservative extension of
$\langle \n, A\rangle_{A\in\x}$ where $A'=\{[f]_G\in \Pi_\x\n/G\mid
\{n\in\n\mid f(n)\in A\}\in G\}$. Fix a set $E$ definable in
$\la\Pi_\x\n/G, A'\ra_{A\in\x}$ by a formula $\varphi(x,[f]_G)$.
Observe that $n\in E\leftrightarrow \Pi_\x\n/G\models
\varphi(n,[f]_G)\leftrightarrow B_n^{\varphi,f}=\{m\in \n\mid
\n\models \varphi(n,f(m))\}\in G$. Let $\mathscr
D_{\varphi,f}=\break\{C\in \x/\fin\mid C \text{ decides }\langle
B_n^{\varphi, f}\mid n\in\n\rangle\}$.  The sets $\mathscr
D_{\varphi, f}$ are dense since $\x$ is decisive. Clearly if $G$
meets all the $\mathscr D_{\varphi,f}$, the ultrapower $\la
\Pi_\x\n/G, A'\ra_{A\in\x}$ will be a conservative extension of
$\langle \n, A\rangle_{A\in\x}$. Finally, since $\x$ has size
$\omega_1$, there are at most $\omega_1$ many formulas $\varphi$ of
$\lan_\x$ and functions $f$ coded in $\x$, and hence at most
$\omega_1$ many dense sets $\mathscr D_{\varphi,f}$. So we can find
the desired $G$ by {\rm PFA}.\footnote{The anonymous referee pointed
out that similar arguments have appeared in \cite{blass:pa}.}
\end{proof}

Another open question in the field of models of {\rm PA}, for which
$\x/\fin$ is relevant, involves the existence of minimal cofinal
extensions for uncountable models.
\begin{definition}
Let $M$ be a model of {\rm PA}, then $M\prec N$ is a \emph{minimal
extension} if it is a proper extension and whenever $M\prec K\prec
N$, either $K=M$ or $K=N$.
\end{definition}
\begin{theorem}
Every nonstandard countable model of {\rm PA} has a minimal cofinal
extension. \emph{(See \cite{kossak:book}, p. 28)}
\end{theorem}

Gaifman showed that every model of {\rm PA}, regardless of
cardinality, has a minimal end extension \cite{gaifman:types}.
\begin{definition}
Let $\x\subseteq \power(\n)$ be a Boolean algebra.  If $U$ is an
ultrafilter on $\x$, we say that $U$ is \emph{Ramsey} if for every
function $f:\n\to\n$ coded in $\x$, there is a set $A\in U$ such
that $f$ is either $1\text{-1}$ or constant on $A$.
\end{definition}
\begin{lemma} If $M$ is a nonstandard model of {\rm PA} such that
${\rm SSy}(M)$ has a Ramsey ultrafilter, then $M$ has a minimal
cofinal extension.\footnote{This was first proved by
\cite{blass:pa}.}
\end{lemma}
\begin{proof}
Let $U$ be a Ramsey ultrafilter on ${\rm SSy}(M)$. The strategy will
be to show that the ultrapower $\Pi_{{\rm SSy}(M)}M/U$ is a minimal
cofinal extension of $M$. The meaning of $\Pi_{{\rm SSy}(M)}M/U$
here is identical to the one in the proof of Theorem \ref{th:ehren}.
First, observe that for any ultrafilter U, we have $\Pi_\x
M/U=Scl(M\cup\{[id]_U\})$, the Skolem closure of the equivalence
class of the identity function together with elements of $M$. This
holds since any $[f]_U=t([id]_U)$ where $t$ is the Skolem term
defined by $f$ in $M$. Next, observe that such ultrapowers are
always cofinal. To see this, fix $[f]_U\in\Pi_\x M/U$ and let
$a>f(n)$ for all $n\in\n$. Such $a$ exists since $f$ is coded in
$M$. Clearly $[f]_U<[c_a]_U$ where $c_a(n)=a$ for all $n\in\n$.
These observations hold for any Scott set $\x\supseteq {\rm SSy}(M)$
and, in particular, for $\x={\rm SSy}(M)$. To show that the
extension $\Pi_{{\rm SSy}(M)}M/U$ is minimal, we fix $M\prec K\prec
\Pi_\x M/U$ and show that $K=M$ or $K=\Pi_\x M/U$. It suffices to
see that $[id]_U\in Scl(M\cup\{[f]_U\})$ for every $[f]_U\in
(\Pi_{{\rm SSy}(M)} M/U)- M$. Fix $f:\n\to M$ and define $g:\n\to\n$
such that $g(0)=0$ and $g(n)=n$ if $f(n)$ is not equal to $f(m)$ for
any $m<n$, or $g(n)=m$ where $m$ is least such that $f(m)=f(n)$.
Observe that $g\in {\rm SSy}(M)$. Also for any $A\subseteq \n$, the
function $g$ is $1\text{-}1$ or constant on $A$ if and only if $f$
is. Since $U$ is Ramsey, $g$ is either constant or $1\text{-}1$ on
some set $A\in U$. Hence $f$ is either constant or $1\text{-}1$ on
$A$ as well. If $f$ is constant on $A$, then $[f]_U\in M$.  If $f$
is $1\text{-}1$ on $A$, let $s$ be the Skolem term that is the
inverse of $f$ on $A$. Then clearly $s([f]_U)=[id]_U$. This
completes the argument that $\Pi_{{\rm SSy}(M)} M/U$ is a minimal
cofinal extension of $M$.
\end{proof}
The converse to the above theorem does not hold. If $M$ has a
minimal cofinal extension, it does not follow that there is a Ramsey
ultrafilter on ${\rm SSy}(M)$.\footnote{I am grateful to Haim
Gaifman for pointing this out, see \cite{gitman:thesis} for a
detailed argument.}
\begin{theorem}\label{th:ramseyult}
Assuming {\rm PFA}, Ramsey ultrafilters exist for proper Scott sets
of size $\omega_1$. Thus, if  $M$ is a model of {\rm PA} and ${\rm
SSy}(M)$ is proper of size $\omega_1$, then $M$ has a minimal
cofinal extension.
\end{theorem}
\begin{proof}
The existence of a Ramsey ultrafilter involves being able to meet a
family of dense sets.  To see this, fix $f:\n\to\n$ and observe that
$\mathscr D_f=\{A\in {\rm SSy}(M)/\fin\mid f \text{ is 1-1 on
}A\text{ or }f\text{ is constant on }A\}$ is dense. To see that
$\mathscr D_f$ is dense, actually does not require that ${\rm
SSy}(M)$ is arithmetically closed.
\end{proof}
The proof of Theorem \ref{th:ramseyult} shows that any $M$ with a
countable standard system will have a minimal cofinal extension
since we do not need {\rm PFA} to construct an ultrafilter meeting
countably many dense sets.

\section{Weakening the Hypothesis}
There are several ways in which the hypothesis of the main theorem
can be modified. {\rm PFA} is a very strong set theoretic axiom, and
therefore it is important to see whether this assumption can be
weakened to something that is lower in consistency strength. In
fact, there are weaker versions of {\rm PFA} that still work with
the main theorem. It is also possible to make slightly different
assumptions on $\x$. Instead of assuming that $\x$ is proper, it is
sufficient to assume that $\x$ is the union of a chain of proper
Scott sets.

The definition of properness refers to countable structures $M\prec
H_\lambda$ and the existence of $M$-generic elements for them.  If
we fix a cardinal $\kappa$ and modify the definition to consider $M$
of size $\kappa$ instead, we will get the notion of
$\kappa$\emph{-properness}. In this extended definition, the notion
of properness we considered up to this point becomes
$\aleph_0$\emph{-properness}. For example, the $\kappa$-c.c.\ and
$<\kappa$-closed posets are $\kappa$-proper. Hamkins and Johnstone
\cite{johnstone:pfa} recently proposed a new axiom ${\rm
PFA}(\mathfrak{c}\text{-proper})$ which states that for every poset
$\p$ that is proper and $2^\omega$-proper and every collection
$\mathcal D$ of $\omega_1$ many dense subsets of $\p$, there is a
filter on $\p$ that meets all of them. ${\rm
PFA}(\mathfrak{c}\text{-proper})$ is much weaker in consistency
strength than {\rm PFA}. While the best large cardinal upper bound
on the consistency strength of {\rm PFA} is a supercompact cardinal,
an upper bound for ${\rm PFA}(\mathfrak{c}\text{-proper})$ is an
unfoldable cardinal \cite{johnstone:pfa}. Unfoldable cardinals were
defined by Villaveces \cite{villaveces:unfoldable} and are much
weaker than measurable cardinals.  In fact, unfoldable cardinals are
consistent with $V=L$. The axiom ${\rm
PFA}(\mathfrak{c}\text{-proper})$ also decides the size of the
continuum is $\omega_2$ \cite{johnstone:pfa}. It is enough for the
main theorem to assume that ${\rm PFA}(\mathfrak{c}\text{-proper})$
holds:
\begin{theorem}
Assuming ${\rm PFA}(\mathfrak{c}\text{-proper})$, every proper Scott
set is the standard system of a model of {\rm PA}.
\end{theorem}
\begin{proof}
Every $\kappa^+$-c.c.\ poset is $\kappa$-proper. It is clear that
every Scott set $\x$ is $(2^\omega)^+$-c.c.. It follows that every
Scott set is $2^\omega$-proper. Thus, ${\rm
PFA}(\mathfrak{c}$-$\text{proper})$ applies to proper Scott sets.
\end{proof}
It is also easy to see that we do not need the whole Scott set $\x$
to be proper.  For the construction, it would suffice if $\x$ was a
union of a chain of proper Scott sets. Call a Scott set
\emph{piecewise proper} if it is the union of a chain of proper
Scott sets of size $\leq\omega_1$. Under this definition, any
arithmetically closed Scott set of size $\leq\omega_1$ is trivially
piecewise proper since it is the union of a chain of arithmetically
closed countable Scott sets. Also, it is clear that a piecewise
proper Scott set is arithmetically closed. The modified construction
using piecewise proper Scott sets does not require all of {\rm PFA}
but only a much weaker version known as ${\rm PFA}^-$. The axiom
${\rm PFA}^-$ is the assertion that for every proper poset $\p$ of
\emph{size} $\omega_1$ and every collection $\mathcal D$ of
$\omega_1$ many dense subsets of $\p$, there is a filter on $\p$
that meets all of them. ${\rm PFA}^-$ has no large cardinal
strength. The axiom is equiconsistent with {\rm ZFC}
\cite{shelah:proper} (p. 122). This leads to the following modified
version of the main theorem:
\begin{theorem}\label{th:mainbpfa}
Assuming ${\rm PFA}^-$, every piecewise proper Scott set of size
$\leq\omega_2$ is the standard system of a model of {\rm PA}.
\end{theorem}
\begin{proof}
It suffices to show that the $\omega_2$-Ehrenfeucht Principle holds
for piecewise proper Scott sets of size $\omega_2$. So suppose $M$
is a model of {\rm PA} of size $\omega_1$ and $\x$ is a piecewise
proper Scott set of size $\omega_2$ such that ${\rm SSy}(M)\subseteq
\x$.  Since $\x$ is piecewise proper, it is the union of a chain of
proper Scott sets $\x_\xi$ for $\xi<\omega_2$. Fix any $A\in \x$,
then there is an ordinal $\alpha<\omega_2$ such that ${\rm SSy}(M)$
and $A$ are contained in $\x_\alpha$. Since $\x_\alpha$ is proper,
the $\omega_2$-Ehrenfeucht Principle holds for $\x_\alpha$ by
Theorem \ref{th:ehren}.  Thus, there is $M\prec N$ such that
$A\in{\rm SSy}(N)\subseteq \x_\alpha\subseteq \x$.
\end{proof}

\section{When is $\x/\fin$ Proper or Piecewise Proper?}\label{sec:proper}
Here, I give an overview of what is known about the existence of
proper and piecewise proper Scott sets. Recall that for a  property
of posets $\mathscr P$, if $\x$ is arithmetically closed and
$\x/\fin$ has $\mathscr P$, I say that \emph{$\x$ has property
$\mathscr P$}.
\begin{theorem}
Any arithmetically closed countable Scott set is proper and
$\power(\n)$ is proper.
\end{theorem}
\begin{proof}
The class of proper posets includes c.c.c.\ and countably closed
posets.  An arithmetically closed countable Scott set is c.c.c. and
$\power(\n)$ is countably closed.
\end{proof}
We are already in a better position than with c.c.c.\ Scott sets
since we have an instance of an uncountable proper Scott set, namely
$\power(\n)$. This does not, however, give us a new instance of
Scott's Problem since we already know by the Compactness Theorem
that there are models of {\rm PA} with standard system $\power(\n)$.

The easiest way to show that a poset is proper is to show that it is
c.c.c.\ or countably closed. We already know that if a Scott set is
c.c.c., then it is countable (Theorem \ref{th:wrong}). So this
condition gives us no new proper Scott sets. It turns out that
neither does the countably closed condition.
\begin{theorem}  If $\x$ is any Scott set such that $\x/\fin$ is countably closed,
then $\x=\power(\n)$.
\end{theorem}
\begin{proof}
First, I claim that if $\x/\fin$ is countably closed, then $\x$ is
arithmetically closed. I will show that for every sequence $\la
B_n\mid n\in\n\ra$ coded in $\x$, there is $C\in\x$ deciding
\hbox{$\la B_n\mid n\in\n\ra$}. This suffices by Theorem
\ref{th:decisive}. Fix \hbox{$\la B_n\mid n\in\n\ra$} coded in $\x$.
Define a descending sequence $B_0^*\geq B_1^*\geq\cdots\geq
B_n^*\geq\cdots$ of elements of $\x/\fin$ by induction on $n$ such
that $B_0^*=B_0$ and $B_{n+1}^*$ is $B_n^*\cap B_{n+1}$ if this
intersection is infinite or $B_n^*\cap (\n-B_{n+1})$ otherwise. By
countable closure, there is $C\in \x/\fin$ below this sequence.
Clearly $C$ decides $\la B_n\mid n\in\n\ra$. Therefore $\x$ is
arithmetically closed. Now I will show that every $A\subseteq \n$ is
in $\x$. Define $B_n=\{m\in\n\mid (m)_n=\chi_A(n)\}$ as before. Let
$A_m=\cap_{n\leq m}B_n$ and observe that $A_0\geq A_1\geq\cdots \geq
A_m\geq\ldots$ in $\x/\fin$. By countable closure, there exists
$C\in \x/\fin$ such that $C\subseteq_{\fin}A_m$ for all $m\in \n$.
Thus, $C\subseteq_\fin B_n$ for all $n\in \n$. It follows that
$A=\{n\in\n\mid \break \exists m\,\forall k\in C\,\text{ if }k>m,
\text{ then } (k)_n=1\}$. Thus, $A$ is arithmetic in $C$, and hence
$A\in\x$ by arithmetic closure. Since $A$ was arbitrary, this
concludes the proof that $\x=\power(\n)$.
\end{proof}
The countable closure condition can be weakened slightly. If a poset
is just strategically $\omega$-closed, it is enough to imply
properness.
\begin{definition}
Let $\p$ be a poset, then $\mathscr{G}_{\p}$ is the following
infinite game between players I and II:  Player I plays an element
$p_0\in\p$, and then player II plays $p_1\in \p$ such that $p_0\geq
p_1$. Then player I plays $p_1\geq p_2$ and player II plays $p_2\geq
p_3$. Player I and II alternate in this fashion for $\omega$ steps
to end up with the descending sequence $p_0\geq p_1\geq
p_2\geq\ldots\geq p_n\geq \ldots$.  Player II wins if the sequence
has a lower bound in $\p$. Otherwise, player I wins. A poset $\p$ is
\emph{strategically $\omega$-closed} if player II has a winning
strategy in the game $\mathscr{G}_{\p}$.
\end{definition}
Observe that if $\x$ is a Scott set such that $\x/\fin$ is
strategically $\omega$-closed, then $\x$ has to be arithmetically
closed. To see this, suppose that $\x/\fin$ is strategically
$\omega$-closed and $\la B_n\mid n\in\n\ra$ is a sequence coded in
$\x$. We will find $C\in\x/\fin$ deciding the sequence by having
player I play either $B_n$ or $\n-B_n$ intersected with the previous
move of player II at the $n^{\text{th}}$ step of the game. It is not
known whether there are Scott sets that are strategically
$\omega$-closed but not countably closed.

One might wonder at this point whether it is possibly the case that
a Scott set is proper only when it is countable or $\power(\n)$ and
a Scott set is piecewise proper only when it is of size
$\leq\omega_1$. In a forthcoming paper \cite{gitman:proper}, I show
the following results about the existence of proper and piecewise
proper Scott sets.

First, I show that one can obtain uncountable proper Scott sets
other that $\power(\n)$ by considering when the $\power(\n)$
\emph{of} $V$ remains proper in a generic extension after forcing to
add new reals.
\begin{theorem}
If {\rm CH} holds and $\p$ is a c.c.c.\ poset, then
$\power^V(\n)/\fin$ remains proper in $V[g]$ where $g\subseteq \p$
is $V$-generic.
\end{theorem}

In particular, if {\rm CH} holds in $V$ and we force to add a Cohen
real, then the $\power(\n)$ \emph{of} $V$ will be a proper Scott set
in the generic extension.

It is also possible to force the existence of \emph{many} proper
Scott sets of size $\omega_1$ and piecewise proper Scott sets of
size $\omega_2$.

\begin{theorem}
There is a generic extension of $V$ by a c.c.c.\ poset, which
contains continuum many proper Scott sets of size $\omega_1$.
\end{theorem}

\begin{theorem}
There is a generic extension of $V$ by a c.c.c.\ poset, which
contains continuum many piecewise proper Scott sets of size
$\omega_2$.
\end{theorem}

Finally, Enayat showed in \cite{enayat:endextensions} that {\rm ZFC}
proves the existence of an arithmetically closed Scott set of size
$\omega_1$ which is \emph{not} proper.
\begin{theorem}[Enayat, 2006]
There is an arithmetically closed Scott set $\x$ such that $\x/\fin$
collapses $\omega_1$. Hence $\x$ is not proper.
\end{theorem}
Clearly $\x/\fin$ cannot be proper since proper posets preserve
$\omega_1$.

Recall that any arithmetically closed Scott set of size $\omega_1$
is trivially piecewise proper. It follows that there are piecewise
proper Scott sets which are not proper. It is not clear whether
every proper Scott has to be piecewise proper.
\section{Questions}
\begin{question}
Can {\rm ZFC} or {\rm ZFC} + {\rm PFA} prove the existence of an
uncountable proper Scott set other than $\power(\n)$?
\end{question}
\begin{question}
Is it consistent with {\rm ZFC} that there are proper Scott sets of
size $\omega_2$ other than $\power(\n)$?
\end{question}
\begin{question}
Are there Scott sets that are strategically $\omega$-closed but not
countably closed?
\end{question}
\begin{question}
Does the $\omega_2$-Ehrenfeucht Principle hold or fail
(consistently)?
\end{question}
\begin{question}
Does the $\omega_2$-Ehrenfeucht Principle hold for models with a
countable standard system? That is, can we remove the assumption of
\emph{arithmetic closure} from Theorem \ref{th:countehren}?
\end{question}
\bibliographystyle{alpha}
\bibliography{database}
\end{document}